\documentclass{amsart}[12pt]
\usepackage{amsmath, amsthm, amscd, amsfonts,}
\usepackage{graphicx,float}

\usepackage{amssymb}

\usepackage{pb-diagram}
\usepackage{lamsarrow}
\usepackage{pb-lams}
\usepackage{mathrsfs}
\usepackage{eufrak}
\DeclareMathOperator{\SCH}{SCH}
\DeclareMathOperator{\ZFC}{ZFC}

\DeclareMathOperator{\range}{range}
\DeclareMathOperator{\cf}{cf}

\def\k{\kappa}

\def\l{\lambda}

\newtheorem{theorem}{Theorem}[section]

\newtheorem{lemma}[theorem]{Lemma}

\newtheorem{definition}[theorem]{Definition}

\newtheorem{remark}[theorem]{Remark}
\newtheorem{claim}[theorem]{Claim}
\numberwithin{equation}{section}

\usepackage{pb-diagram}

\def\l{\lambda}

\def\rmark{\mbox{$\rm\bf\rule{0.06em}{1.45ex}\kern-0.05em R$}}
\def\pmark{\mbox{$\rm\bf\rule{0.06em}{1.45ex}\kern-0.05em P$}}
\def\nmark{\mbox{$\rm\bf\rule{0.06em}{1.45ex}\kern-0.05em N$}}
\def\vdash{\mbox{$\rm\| \kern-0.13em -$}}

\begin{document}

\title[On cuts in ultraproducts of linear orders II]{On cuts in ultraproducts of  linear orders II}

\author[ M. Golshani and S. Shelah]{ Mohammad Golshani and Saharon Shelah}

\thanks{The first author's research has been supported by a grant from IPM (No. 91030417). The second
author's research has been partially supported by the European Research Council grant 338821. This is
publication 1087 of second author.}
\thanks{The authors thank the referee of the paper for his useful comments and suggestions.}\maketitle

\begin{abstract}
We continue our study  of the class $\mathscr{C}(D)$, where $D$ is a uniform ultrafilter on a cardinal $\kappa$
and  $\mathscr{C}(D)$ is the class of all pairs $(\theta_1, \theta_2),$ where  $(\theta_1, \theta_2)$ is the cofinality of  a cut in $J^\k/D$ and $J$ is some  $(\theta_1+\theta_2)^+$-saturated dense linear order.
We give a combinatorial characterization   of the class $\mathscr{C}(D)$. We also show that if $(\theta_1, \theta_2) \in \mathscr{C}(D)$ and $D$ is $\aleph_1$-complete or $\theta_1 + \theta_2 > 2^\kappa,$ then $\theta_1=\theta_2.$
\end{abstract}
\maketitle

\section{Introduction}
Assume $\kappa$ is an infinite cardinal and $D$ is an ultrafilter on $\kappa$.  Recall that $\mathscr{C}(D)$ is defined to be the class of all pairs $(\theta_1, \theta_2),$ where  $(\theta_1, \theta_2)$ is the cofinality of  a cut in $J^\k/D$ and $J$ is some (equivalently any) $(\theta_1+\theta_2)^+$-saturated dense linear order. Also $\mathscr{C}_{> \lambda}(D)$ is defined to be the class of all pairs $(\theta_1, \theta_2) \in \mathscr{C}(D),$ such that $\theta_1 + \theta_2 > \lambda$. The classes $\mathscr{C}_{\geq \lambda}(D), \mathscr{C}_{< \lambda}(D)$ and $\mathscr{C}_{\leq \lambda}(D)$ are defined similarly.

The works \cite{malliaris-shelah1}, \cite{malliaris-shelah2} and \cite{malliaris-shelah3} of Malliaris and Shelah have started  the study of this class for the case $\theta_1 + \theta_2 \leq 2^\k$ and \cite{golshani-shelah} started the study of the case $\theta_1 + \theta_2 > 2^\k$.
As it was observed in \cite{golshani-shelah}, the study of the class $\mathscr{C}_{> 2^{\kappa}}(D)$
is very different from the case $\mathscr{C}_{\leq 2^{\kappa}}(D),$ and to prove results about it, usually some extra set theoretic assumptions are needed. In this paper we continue
 \cite{golshani-shelah} and prove more results related to  the class $\mathscr{C}(D)$.

In the first part of the paper (Sections 2 and 3) we give a combinatorial characterization of $\mathscr{C}(D)$. Using notions defined in section 2, we can state our first main theorem as follows.
\begin{theorem}
Assume
$D$ is an ultrafilter  on   $\kappa$ and $\lambda_1, \lambda_2 > \kappa$ are regular cardinals. The following are equivalent:

$(a)$ There is $\bar{a} \in \mathcal{S}_c$ which is not $c$-solvable, where $c= \langle \kappa, D, \lambda_1, \lambda_2 \rangle.$

$(b)$ $( \lambda_{1}, \lambda_{2}) \in \mathscr{C}(D)$.
\end{theorem}

In the second part of the paper (Sections 4 and 5) we study the existence of non-symmetric pairs (i.e., pairs $(\l_1, \l_2)$ with $\l_1 \neq \l_2$) in  $\mathscr{C}(D).$
By \cite{shelah1026}, we can find a regular ultrafilter $D$ on $\kappa$ such that
\[
\mathscr{C}(D) \supseteq \{ (\lambda_1, \lambda_2): \aleph_0 < \lambda_1 < \lambda_2 \leq 2^\kappa, \l_1, \l_2 \text{~regular}         \}.
\]
In particular, $\mathscr{C}(D)$ contains non-symmetric pairs.
On the other hand, results of \cite{golshani-shelah} show that if $(\l_1, \l_2) \in \mathscr{C}_{> 2^\k}(D)$, then we must have $\l_1^\k = \lambda_2^\k,$
in particular if $\SCH$, the singular cardinals hypothesis, holds, then $\l_1 = \l_2,$ and so $\mathscr{C}_{> 2^\k}(D)$ just contains symmetric pairs.
We then prove the following theorem (in $\ZFC$):
\begin{theorem}
$(a)$ Assume $D$ is a uniform $\aleph_1$-complete ultrafilter on $\k$ and $(\l_1, \l_2) \in \mathscr{C}(D).$
Then $\l_1=\l_2.$

$(b)$ Assume $D$ is a uniform  ultrafilter on $\k$ and $(\l_1, \l_2) \in \mathscr{C}_{> 2^\k}(D).$
Then $\l_1=\l_2.$
\end{theorem}
The theorem shows some restrictions on the pairs $(\l_1, \l_2)$ that
 $\mathscr{C}(D)$ can have, in particular, it shows that in the result of \cite{shelah1026} stated above, we can never take the ultrafilter $D$ to be
 $\aleph_1$-complete and that $\mathscr{C}_{> 2^\k}(D)$ can not have non-symmetric pairs.

 The paper is organized as follows. In section 2 we give the required definitions, which lead us to the notion of $c$-solvability and in section 3 we complete the proof of Theorem 1.1.
 In section 4 we prove part $(a)$ of  Theorem 1.2 and in section 5 we complete the proof of part $(b)$ of  Theorem 1.2.
We may note that parts one (Sections 2 and 3) and two (Sections 4 and 5) can be read independently of each other.
\section{On the notion of $c$-solvability}
In this section we give the required definitions which are used in Theorem 1.1.

\begin{definition}
\begin{enumerate}
\item [(a)] Let $\mathcal{C}$ be the class of tuples $c=\langle  \kappa_c, D_c, \lambda_{c,1}, \lambda_{c,2}       \rangle$ where
\begin{enumerate}
\item [(a-1)] $\lambda_{c,1}, \lambda_{c,2}$ are regular cardinals $> \kappa_c,$
\item [(a-2)] $D_c$ is a uniform ultrafilter on $\kappa_c.$
\end{enumerate}
Also let $\lambda_c=2^{<\lambda_{c,1}}+ 2^{<\lambda_{c,2}}$ and $\lambda_{c,0}=\min\{\lambda_{c,1}, \lambda_{c,2} \}.$

\item [(b)] For $c \in \mathcal{C}$ let $N_c = N_{c,1}+N_{c,2}$ be a linear order
of size $\leq \lambda_c$
in such a way that $N_{c,1}$ has cofinality $\lambda_{c,1}$, $N_{c,2}$ has co-initiality
$\lambda_{c,2}$ and both $N_{c,1}, N_{c,2}$ are  $\lambda_{c,0}$-saturated dense linear orders \footnote{$N_c$ is some fixed linear order which we choose in advance. We may assume global choice and let $N_c$ be the least such order.}.

\item [(c)]  For $c \in \mathcal{C}$ let $\mathcal{S}_c$ be the set of all sequences $\bar{a}= \langle   a_{s,t}: s, t \in N_c     \rangle$ such that
\begin{enumerate}
\item [(c-1)] Each $a_{s,t}$ is a subset of $\kappa_c,$
\item [(c-2)] $a_{s,s}=\emptyset,$
\item [(c-3)] For $s \neq t, a_{s,t}=\kappa \setminus a_{t,s},$
\item [(c-4)] $s <_{N_c}t \Rightarrow a_{s,t} \in D_c,$
\item [(c-5)] If $s_1 <_{N_c} s_2 <_{N_c} s_3$, then
\[
(a_{s_1, s_3} \supseteq a_{s_1, s_2} \cap a_{s_2, s_3}) ~ \&~ (a_{s_3, s_1} \supseteq a_{s_3, s_2} \cap a_{s_2, s_1}).
\]
\end{enumerate}
\item [(d)]  For $c \in \mathcal{C}$ let $N_c^+=N_{c,1} + N_0 + N_{c,2},$ where $N_0$ is a singleton, say $N_0=\{  s_* \}.$
\end{enumerate}
\end{definition}
We now define the notion of $c$-solvability.
\begin{definition}
Let $c \in \mathcal{C}$. We say $\bar{a} \in \mathcal{S}_c$ is $c$-solvable, if there exists a sequence $\bar{b}= \langle b_s: s \in N_c   \rangle,$ such that the sequence
$\bar{a}^1= \bar{a}*\bar{b}$ satisfies clauses $(c$-$1)$-$(c$-$5)$ above, where the sequence $\bar{a}^1= \langle   a^1_{s,t}: s, t \in N_c^+     \rangle$
is defined as follows:
\begin{enumerate}
\item If $s,t \in N_c,$ then $a^1_{s,t}=a_{s,t}$,
\item For $s \in N_{c,1}, a^1_{s, s_*}=b_s$ and $a^1_{s_*, s}=\kappa_c \setminus b_s,$
\item For $s \in N_{c,2}, a^1_{s_*, s}=b_s$ and $a^1_{s, s_*}=\kappa_c \setminus b_s,$
\item $a^1_{s_*, s_*} =\emptyset.$
\end{enumerate}
Then $\bar{b}$ is called a $c$-solution for $\bar{a}$.
\end{definition}

\section{A combinatorial characterization of $\mathscr{C}(D)$}
In this section we give a proof of Theorem 1.1.
\begin{lemma}
Assume $c \in \mathcal{C}$ and  $\bar{a} \in \mathcal{S}_c$. Then
\begin{enumerate}
\item [(a)]  There are $M, \bar{f}$ such that
\begin{enumerate}
\item [(a-1)] $M$ is a $(\lambda_{c,1}+\lambda_{c,2})^+$-saturated dense linear order,

\item [(a-2)] $\bar{f}=\langle f_s: s \in N_c      \rangle,$

\item [(a-3)] Each $f_s \in$$~ ^{\kappa_c}$$M$,

\item [(a-4)] If $s <_{N_c} t,$ then $a_{s,t} = \{i< \kappa_c: f_s(i) <_M f_t(i)    \}$,

\item [(a-4)] $\langle \range(f_s): s \in N_c    \rangle$
is a sequence of pairwise disjoint sets.
\end{enumerate}
\item [(b)] If $M, \bar{f}$ are as in $($a$)$, then
\begin{enumerate}
\item [(b-1)] $\langle  f_s / D_c: s \in N_c    \rangle$ is an increasing sequence in $^{\kappa_c}$$M/D_c$,

\item [(b-2)] $\bar{a}$ is $c$-solvable iff  $^{\kappa_c}$$M/D_c$ realizes the type
\[
q(x)=\{ f_s / D_c < x < f_t / D_c: s \in N_{c,1} \text{~and~} t \in N_{c,2}        \}.
\]
\end{enumerate}
\end{enumerate}
\end{lemma}
\begin{proof}
$($a$)$ Let $A=\{ (i, s): i < \kappa_c, s \in N_c   \}$, and define the order $<_A$ on $A$ by
\[
(i_1, s_1) <_A (i_2, s_2) \iff (i_1 < i_2) \text{~or~} (i_1=i_2 \in a_{s_1, s_2}).
\]
Also let $\leq_A$ be defined on $A$ in the natural way from $<_A$, so
  \[
(i_1, s_1) \leq_A (i_2, s_2) \iff (i_1, s_1) = (i_2, s_2)  \text{~or~} (i_1, s_1) <_A (i_2, s_2) .
\]
It is easily seen that $\leq_A$ is a linear order on $A$. Now let $M$ be a $(\lambda_{c,1}+\lambda_{c,2})^+$-saturated dense linear order which contains $(A, <_A)$
as a sub-order. Also let  $\bar{f}=\langle f_s: s \in N_c      \rangle,$ where for $s \in N_c$  $f_s \in ~$$^{\kappa_c}$$M$ is defined by $f_s(i)=(i, s).$
It is clear that  $M,$ and $\bar{f}$ satisfy clauses $($a-1$)$-$($a-3$)$. For $($a-4$)$, assume $s <_{N_c} t$ are given. Then
\begin{center}
$a_{s,t}=\{ i < \kappa_c: i \in a_{s,t}       \}=\{i < \kappa_c: (i, s) <_A (i, t)    \}= \{i< \kappa_c: f_s(i) <_M f_t(i)    \}.$
\end{center}
Finally note that for $s \neq t$ in $N_c$,
\[
\range(f_s) \cap \range(f_t) = \{(i, s): i < \kappa_c     \} \cap \{(i, t): i < \kappa_c     \}=\emptyset.
\]
So $M$ and $\bar{f}$ are as required.

$($b$)$ $($b-$1)$ follows from $($a-$4)$ and the fact that for $s <_{N_c} t, a_{s,t} \in D_c$. Let's prove $($b-$2)$. First assume that $\bar{a}$ is $c$-solvable and let $\bar{b}$ be a solution for $\bar{a}.$
For each $i<\kappa_c$ let $p_i(x)$ be the following type over $M$:
\begin{center}
$p_i(x)=\{ f_s(i) <_M x: s \in N_{c,1}$ and $i \in b_s   \} \cup \{ x <_M f_t(i): t \in N_{c,2}$ and $i \in \kappa_c \setminus b_t   \}$.
\end{center}
\begin{claim}
For each $i<\kappa_c,$ the type $p_i(x)$ is finitely satisfiable in $M$.
\end{claim}
\begin{proof}
Let $s_0 <_{N_{c,1}} \dots <_{N_{c,1}} s_{n-1}$ be in $N_{c,1}$ and $t_{m-1} <_{N_{c,2}} < \dots <_{N_{c,2}} t_0$ be in $N_{c,2}$. Also suppose that $i \in \bigcap_{k < n} b_{s_k} \cap \bigcap_{l < m} (\kappa_c \setminus b_{t_l})$. Then for $k<n$ and $l<m$ we have
\[
a_{s_k, t_l} \supseteq a^1_{s_k, s_*} \cap a^1_{s_*, t_l} = b_{s_k} \cap (\kappa_c \setminus b_{t_l}),
\]
and so $i \in a_{s_k, t_l}$, which implies $f_{s_k}(i) < f_{t_l}(i).$  Take $x \in M$ so that
\[
\forall k<n, \forall l<m,~ f_{s_k}(i)< x < f_{t_l}(i),
\]
which exists as $M$ is dense. It follows that $p_i(x)$ is finitely satisfiable in $M$.
\end{proof}
It follows that there exists $f \in$$~ ^{\kappa_c}$$M$ such that for each $i<\kappa_c, f(i)$ realizes the type $p_i(x)$ over $M$. Then $f/D_c$ realizes $q(x)$
over $^{\kappa_c}$$M/D_c$.

Conversely assume that $f \in$$~^{\kappa_c}$$M$ is such that $f/D_c$ realizes the type $q(x)$ over $^{\kappa_c}$$M/D_c$.
\begin{claim}
We can assume that $\range(f)$ is disjoint from $A$.
\end{claim}
\begin{proof}
As $\langle \range(f_s): s \in N_c    \rangle$
is a sequence of pairwise disjoint sets and $\lambda_{c,1}, \lambda_{c,2} > \kappa_c$ are regular, there are $s_1 \in N_{c,1}$ and $s_2 \in N_{c,2}$
such that $s_1 <_{N_c} s <_{N_c} s_2$ implies $\range(f_s) \cap \range(f) =\emptyset$. As $M$ is
a $(\lambda_{c,1}+\lambda_{c,2})^+$-saturated dense linear order, there is $f'$ such that
\begin{itemize}
\item $f' \in ~^{\kappa_c}$$M$,
\item $\range(f') \cap A =\emptyset,$
\item If $s_1 <_{N_c} s <_{N_c} s_2$ and $i< \kappa_c,$ then $f_s(i) <_{N_c} f'(i) \Rightarrow f_s(i) <_{N_c} f(i)$
and $f'(i) <_{N_c} f_s(i) \Rightarrow f(i) <_{N_c} f_s(i).$
\end{itemize}
So we can replace $f$ by $f'$ and $f'$ satisfies the requirements on $f$; i.e., $f'/D_c$
realizes  $q(x)$ over $^{\kappa_c}$$M/D_c$ and further $\range(f') \cap A =\emptyset$.
\end{proof}
Now define $\bar{b}= \langle b_s: s \in N_c  \rangle$ by
\begin{center}
 $b_s = \left\{ \begin{array}{l}
       \{ i<\kappa_c: f_s(i) <_{N_c} f(i)    \}  \hspace{1.1cm} \text{ if } s \in N_{c,1},\\
       \{ i<\kappa_c: f(i) <_{N_c} f_s(i)    \}  \hspace{1.1cm} \text{ if } s \in N_{c,2}.
 \end{array} \right.$
\end{center}

\begin{claim}
$\bar{b}$ is a $c$-solution for $\bar{a}$.
\end{claim}
\begin{proof}
We show that conditions $($c-$1)$-$($c-$5)$ of Definition 2.1 are satisfied by $\bar{a}^1=\bar{a}*\bar{b}$ (see Definition 2.2). $($c-$1)$ and $($c-$2)$ are trivial and $($c-$3)$ follows from the fact that $\forall i<\kappa_c, f_s(i) \neq f(i)$ (as $\range(f) \cap A =\emptyset$).

For $($c-$4)$, suppose that $s <_{N_c^+} t$. If both $s,t$
are in $N_c$, then we are done. So suppose otherwise. There are two cases to consider.
\begin{itemize}
\item  If $s=s_*,$ then $t \in N_{c,2}$ and as $f/D_c$ realizes $q(x),$  we have
$f/D_c < f_t/D_c,$ which implies $a^1_{s_*, t}=b_t= \{ i<\kappa_c: f(i) <_{N_c} f_t(i)    \} \in D_c.$

\item
If $t=s_*,$ then $s \in N_{c,1}$ and
as $f/D_c$ realizes $q(x),$  we have
$f_s <_{D_c} f,$ which implies $a^1_{s, s_*}=b_s=\{ i<\kappa_c: f_s(i) <_{N_c} f(i)    \} \in D_c.$
\end{itemize}

For $($c-$5)$, assume $s_1 <_{N_c^+} s_2 <_{N_c^+} s_3$ are in $N_c^+.$ If all $s_1, s_2$ and $s_3$ are in $N_c$, then we are done. So assume otherwise. There are three cases to be considered:
\begin{itemize}
\item If $s_1=s_*,$ then $s_2, s_3 \in N_{c,2},$ and we have

$a^1_{s_*, s_2} \cap a^1_{s_2, s_3} = b_{s_2} \cap a_{s_2, s_3}$

$\hspace{2.2cm}$$= \{i < \kappa_c: (~f(i) <_{N_c} f_{s_2}(i)~) \wedge (~f_{s_2}(i) <_{N_c} f_{s_3}(i)~)     \}$

$\hspace{2.2cm}$$\subseteq  \{ i<\kappa_c: f(i) <_{N_c} f_{s_3}(i)    \}$

$\hspace{2.2cm}$$=b_{s_3}$

$\hspace{2.2cm}$$=a^1_{s_*, s_3}$.

Similarly,

$a^1_{s_3, s_2} \cap a^1_{s_2, s_*} =  a_{s_3, s_2} \cap (\kappa_c \setminus b_{s_2})$

$\hspace{2.2cm}$$= \{i < \kappa_c:  (~f_{s_2}(i) \geq_{N_c} f_{s_3}(i)~) \wedge (~f(i) >_{N_c} f_{s_2}(i)~)     \}$

$\hspace{2.2cm}$$\subseteq  \{ i<\kappa_c: f(i) >_{N_c} f_{s_3}(i)    \}$

$\hspace{2.2cm}$$=\kappa_c \setminus b_{s_3}$

$\hspace{2.2cm}$$=a^1_{s_3, s_*}$.

\item If $s_2=s_*,$ then $s_1 \in N_{c,1}$, $s_3 \in N_{c,2}$ and we have

$a^1_{s_1, s_*} \cap a^1_{s_*, s_3} = b_{s_1} \cap b_{s_3}$

$\hspace{2.2cm}$$= \{i < \kappa_c: (~f_{s_1}(i) <_{N_c} f(i)~) \wedge (~f(i) <_{N_c} f_{s_3}(i)~)     \}$

$\hspace{2.2cm}$$\subseteq  \{ i<\kappa_c: f_{s_1}(i) <_{N_c} f_{s_3}(i)    \}$

$\hspace{2.2cm}$$=a_{s_1, s_3}$.

$\hspace{2.2cm}$$=a^1_{s_1, s_3}$.

Also,

$a^1_{s_3, s_*} \cap a^1_{s_*, s_1} = (\kappa_c \setminus b_{s_3}) \cap (\kappa_c \setminus b_{s_1})$

$\hspace{2.2cm}$$= \{i < \kappa_c:  (~f_{s_3}(i) <_{N_c} f(i)~) \wedge (~f(i) <_{N_c} f_{s_1}(i)~)     \}$

$\hspace{2.2cm}$$\subseteq  \{ i<\kappa_c: f_{s_3}(i) \leq_{N_c} f_{s_1}(i)    \}$

$\hspace{2.2cm}$$=a_{s_3, s_1}$

$\hspace{2.2cm}$$=a^1_{s_3, s_1}$.

\item If $s_3=s_*,$ then  $s_1, s_2 \in N_{c,1}$ and we have

$a^1_{s_1, s_2} \cap a^1_{s_2, s_*} = a_{s_1, s_2} \cap b_{s_2}$

$\hspace{2.2cm}$$= \{i < \kappa_c: (~f_{s_1}(i) <_{N_c} f_{s_2}(i)~) \wedge (~f_{s_2}(i) <_{N_c} f(i)~)     \}$

$\hspace{2.2cm}$$\subseteq  \{ i<\kappa_c: f_{s_1}(i) <_{N_c} f(i)    \}$

$\hspace{2.2cm}$$=b_{s_1}$.

$\hspace{2.2cm}$$=a^1_{s_1, s_*}$.

Similarly, we have

$a^1_{s_*, s_2} \cap a^1_{s_2, s_1} = (\kappa_c \setminus b_{s_2}) \cap a_{s_2, s_1}$

$\hspace{2.2cm}$$= \{i < \kappa_c:  (~f_{s_2}(i) >_{N_c} f(i)~) \wedge (~f_{s_1}(i) \geq_{N_c} f_{s_2}(i)~)     \}$

$\hspace{2.2cm}$$\subseteq  \{ i<\kappa_c: f_{s_1}(i) >_{N_c} f(i)    \}$

$\hspace{2.2cm}$$=\kappa_c \setminus b_{s_1}$

$\hspace{2.2cm}$$=a^1_{s_*, s_1}$.
\end{itemize}
Hence, $\bar{b}$ is a $c$-solution for $\bar{a}$, as required.
\end{proof}
The lemma follows.
\end{proof}
Given $c \in \mathcal{C}$, the next lemma gives a characterization, in terms of $c$-solvability, of when $( \lambda_{c,1}, \lambda_{c,2})$ is in $\mathscr{C}(D_c)$, which
also  completes the proof of Theorem 1.1.
\begin{lemma}
Assume $c \in \mathcal{C}$ and $M$ is a $\lambda_c^+$-saturated dense linear order. The following are equivalent:

$(a)$ There is $\bar{a} \in \mathcal{S}_c$ which is not $c$-solvable.

$(b)$ $( \lambda_{c,1}, \lambda_{c,2}) \in \mathscr{C}(D_c)$.
\end{lemma}
\begin{proof}
First assume there exists $\bar{a} \in \mathcal{S}_c$ which is not $c$-solvable. By Lemma 3.1, there are $M, \bar{f}$
which satisfy clauses $(1$-a$)$-$(1$-e$)$ of that lemma. But then as $\bar{a}$  is not $c$-solvable,
by Lemma 2.3$($b-$2)$, the type
\[
q(x)=\{ f_s / D_c < x < f_t / D_c: s \in N_{c,1} \text{~and~} t \in N_{c,2}        \}.
\]
is not realized. It follows that  $( \lambda_{c,1}, \lambda_{c,2}) \in \mathscr{C}(D_c)$.

Conversely assume that $M$ and $\bar{f}=\langle  f_s: s \in N_c     \rangle$
witness $( \lambda_{c,1}, \lambda_{c,2}) \in \mathscr{C}(D_c)$. Let $A=\bigcup \{ \range(f_s): s \in N_c    \}$,
and let $\langle I_d: d \in A    \rangle$
be a sequence of pairwise disjoint intervals of $M$ such that $d \in I_d$ \footnote{The existence of the sequence $\langle I_d: d \in A    \rangle$ follows from the fact that $|A| \leq \kappa_c \cdot |N_c| \leq \lambda_c$ and $M$ is $\lambda_c^+$-saturated.}. For $s \in N_c,$ let $f'_s \in$$~^{\kappa_c}$$M$
be such that $f'_s(i) \in I_{f_s(i)}$ and $\langle f'_s(i): s \in N_c, i < \kappa      \rangle$
is with no repetitions.
Define the sequence $\bar{a}=\langle  a_{s,t}: s, t \in N_c       \rangle$, such that
for $s <_{N_c}t,$$~a_{s,t}=\{ i<\kappa: f_s(i) < f_t(i)             \}$ and $a_{t,s}=\kappa_c \setminus a_{s,t}$. Also set $a_{s,s}=\emptyset.$
It is evident that
$\bar{a} \in \mathcal{S}_c.$

\begin{claim}
$\bar{a}$ is not $c$-solvable.
\end{claim}
\begin{proof}
Assume not. Then by Lemma 3.1$($b-$2)$, the type
\[
q(x)=\{ f_s / D_c < x < f_t / D_c: s \in N_{c,1} \text{~and~} t \in N_{c,2}        \}
\]
is  realized in $^{\kappa_c}$$M/D_c$, which contradicts
the choice of $M, \bar{f}$.
\end{proof}
The Lemma follows.
\end{proof}

\section{For $\aleph_1$-complete ultrafilter, $\mathscr{C}(D)$ contains no non-symmetric pairs}
In this section we prove part $(a)$ of Theorem 1.2. In fact we will prove something stronger, that is of interest in its own sake.
\begin{definition}
Assume $D$ is an ultrafilter on $\kappa$, $\langle  I_i: i<\kappa     \rangle$ is a sequence of linear orders and $I=\prod_{i<\kappa} I_i/D.$
\begin{enumerate}
\item [$(a)$] a subset $K$ of $I$ is called internal if there are subsets $K_i \subseteq I_i$ such that $K = \prod_{i<\kappa} K_i /D.$

\item [$(b)$] The cut $(J^1, J^2)$ of $I$ is called internal, if there are cuts $(J^1_i, J^2_i)$ of $I_i$, $i<\kappa,$ such that $J^l=\prod_{i<\kappa} J^l_i / D$ ($l=1,2$).
    \end{enumerate}
\end{definition}
\begin{remark}
Assume $J$ is an initial segment of $I$ which is internal and suppose that $(J, I\setminus J)$ is a cut of $I$. Then $(J, I\setminus J)$ is in fact an internal cut of $I$. Similarly, if $J$ is an end segment of $I$ which is internal and if $(I \setminus J, J)$ is a cut of $I$, then $(I \setminus J, J)$ is  an internal cut of $I$.
\end{remark}

\begin{theorem}
Assume $D$ is a uniform $\aleph_1$-complete ultrafilter on $\kappa,$ $\langle  I_i: i<\kappa     \rangle$ is a sequence of non-empty linear orders and $I=\prod_{i<\kappa} I_i/D.$
Also assume  $(J^1, J^2)$ is a cut of $I$ of cofinality $(\theta_1, \theta_2),$ where $\theta_1 \neq \theta_2$. Then the cut $(J^1, J^2)$ is internal.
\end{theorem}
Before giving the proof of Theorem 4.3, let us show that it implies Theorem 1.2$($a$)$.
\\
{\bf Proof of Theorem 1.2$($a$)$ from Theorem 4.3.}
Suppose $D$ is an $\aleph_1$-complete ultrafilter on $\kappa$, $J$ is a $(\lambda_1+\lambda_2)^+$-saturated dense linear order
and  $(J^1, J^2)$ is a cut of $J^\kappa/D$ of cofinality $(\lambda_1, \lambda_2).$ Towards contradiction assume that $\lambda_1 \neq \lambda_2.$ It follows
from Theorem 4.3 that the cut $(J^1, J^2)$ is internal, and so
that there are
cuts $(J^1_i, J^2_i)$ of $J$, $i<\kappa,$ such that $J^l=\prod_{i<\kappa} J^l_i / D$ (for $l=1,2$). Let $(\lambda^1_i, \lambda^2_i)=\cf(J^1_i, J^2_i)$.
It follows that $\lambda_l=\prod_{i<\kappa} \lambda^l_i / D, l=1,2.$

By the choice of $J$, for every $i<\kappa$, either $\lambda^1_i \geq (\lambda_1+\lambda_2)^+$ or $\lambda^2_i \geq (\lambda_1+\lambda_2)^+$, hence for some $l \in \{1,2\},$ we have
\[
A=\{ i< \kappa:   \lambda^l_i \geq (\lambda_1+\lambda_2)^+              \} \in D.
\]
It follows that $\lambda_l=\prod_{i<\kappa} \lambda^l_i / D \geq (\lambda_1+\lambda_2)^+$, which is a contradiction. \hfill$\Box$

We are now ready to complete the proof of Theorem 4.3.
\begin{proof}
We can assume that $\theta_1, \theta_2$ are infinite.  Let
$<^1_i=<_{I_i}$ ($i<\kappa$) and $<_1=<_I.$ Let $<^2_i$ be a well-ordering of $I_i$ with a last element and let $<_2$
be such that $(I, <_2) = \prod_{i<\kappa} (I_i, <^2_i) / D.$ Then $<_2$ is a linear ordering of $I$ with a last element and since $D$ is $\aleph_1$-complete,
 it is well-founded, so  $<_2$ is in fact a well-ordering of $I$ with a last element.

As $(J^1, <_1)$ has cofinality $\theta_1,$ we can find $f_\alpha \in \prod_{i<\kappa}I_i,$ for $\alpha < \theta_1,$
such that
\begin{enumerate}
\item $\forall \alpha < \theta_1, f_\alpha /D \in J^1,$
\item  $\langle f_\alpha/D: \alpha < \theta_1     \rangle$ is $<_1$-increasing,
\item $\langle f_\alpha/D: \alpha < \theta_1     \rangle$ is a $<_1$-cofinal subset of $J^1$.
\end{enumerate}
Let $B=\{t \in I: \{s \in J^1: s <_2 t  \}$ is $<_1$-unbounded in $J^1          \}$. As $\theta_1$ is infinite, the $<_2$-last element of $I$ belongs to $B$, which implies $B \neq \emptyset$
and hence $B$ has a $<_2$-minimal element; call it $t_*.$ Let $g_* \in \prod_{i<\kappa} I_i$ be such that $t_*=g_*/D.$

Note that for each $\alpha< \theta_1$ there are $s \in J^1$ and $\beta > \alpha$ such that  $s <_2 g_*/D$ and $f_\alpha/D <_1 s <_1 f_\beta/D$,
so we can assume that for all
$\alpha < \theta_1, f_\alpha/D <_2 g_*/D.$ This implies
\[
\bigwedge_{\alpha < \theta_1} [ \{ i< \kappa: f_\alpha(i) <^2_i g_*(i)               \} \in D].
\]
Also note that
\[
\{ i<\kappa: g_*(i) \text{~is~} <^1_i\text{-minimal or~} <^1_i\text{-maximal}        \} \notin D,
\]
so, without loss of generality, it is empty. Hence, without loss of generality
\[
\bigwedge_{\alpha < \theta_1} \bigwedge_{i<\kappa} ~[f_\alpha(i) <^2_i g_*(i) \text{~and~} f_\alpha(i)\text{~is not ~}<^1_i\text{-minimal}].
\]
Let $f_{\theta_1}=g_*$ and for $\alpha \leq \theta_1$ set $K_\alpha= \{s \in I: s <_2 f_\alpha /D  \}$. Thus $K_\alpha$
is a $<_2$-initial segment of $I$.
\begin{claim}
$K_\alpha \cap J^1$ is $<_1$-bounded in $J^1$.
\end{claim}
\begin{proof}
As $f_\alpha/D <_2 t_*$, it follows from our choice of $t_*$ that $K_\alpha$ is $<_1$ bounded in $J^1$.
\end{proof}
\begin{claim}
If $\alpha < \theta_1,$ then $K_\alpha$ is an internal subset of $I$.
\end{claim}
\begin{proof}
For each $i<\kappa$ set
\[
K_{\alpha, i}= \{s \in I_i: s <^2_i f_\alpha(i)\}.
\]
Then $K_\alpha = \prod K_{\alpha, i}/D$ and the result follows.
\end{proof}
Now consider the following statement:
\\
$(*)$$\hspace{1.5cm}$ \underline{There is $\alpha < \theta_1$ such that $J^2 \cap K_\alpha$ is $<_1$-unbounded from below in $J^2$.}

We  split the proof into two cases.

{\underline{\bf Case 1. $(*)$ holds:}} Fix $\alpha$ witnessing $(*)$. It follows that $(J^1 \cap K_\alpha, J^2 \cap K_\alpha)$ is internal in $K_\alpha$, so there are end segments
$L_{i}$ of $I_i \upharpoonright \{ s \in I_i: s <^2_i f_\alpha(i)       \}$, for $i<\kappa,$
such that $J^2 \cap K_\alpha = \prod_{i<\kappa} L_i /D,$ hence by the assumption, $J^2 = \prod_{i<\kappa} L_i' /D$,
where $L'_i=\{ t \in I_i: \exists s \in L_i,~ s \leq^1_i t               \}$,
so $J^2$ is internal. It follows from Remark 4.2 that $(J^1, J^2)$ is an internal cut of $I$ and we are done.

{\underline{\bf Case 2. $(*)$ fails:}} So for any $\alpha < \theta_1$, there is $s_\alpha \in J^2$
such that
\[
\{ s \in J^2: s <_1 s_\alpha   \} \cap K_\alpha =\emptyset.
\]
As $\theta_1 \neq \theta_2$ are regular cardinals, there is $s_* \in J^2$ such that
\[
\sup\{\alpha < \theta_1: s_* \leq_1 s_\alpha        \} =\theta_1,
\]
hence
\[
\{ s \in J^2: s <_1 s_*   \} \cap (\bigcup_{\alpha < \theta_1}K_\alpha) =\emptyset.
\]
\begin{claim}
 $\bigcup_{\alpha < \theta_1}K_\alpha = K_{\theta_1}.$
 \end{claim}
 \begin{proof}
 It is clear that $\bigcup_{\alpha < \theta_1}K_\alpha \subseteq K_{\theta_1}.$
 Now suppose $s \in K_{\theta_1},$ so $s <_2  g_*/D.$ If
 $s \notin \bigcup_{\alpha < \theta_1}K_\alpha,$ then   for any $\alpha < \theta_1,$  $f_\alpha/D <_2 s.$ So by the minimal choice of $t_*$ and the fact that $\langle f_\alpha/D: \alpha < \theta_1  \rangle$ is $<_1$-cofinal in $J^1$, we have $g_*/D \leq_2 s$ which is a contradiction.
 \end{proof}
 So we have $\{ s \in J^2: s <_1 s_*   \} \cap K_{\theta_1}=\emptyset.$ Let $h_* \in \prod_{i<\kappa}I_i$
 be such that $s_*=h_*/D.$
\begin{claim}
\begin{enumerate}
\item [(a)] $K_{\theta_1}$ is  internal.

\item [(b)] $J^1 \cap K_{\theta_1}$ is $<_1$-unbounded in $J^1$.

\item [(c)] $J^1 \cap K_{\theta_1}$ is internal.
\end{enumerate}
\end{claim}
\begin{proof}
$($a$)$  can be proved as in Claim 4.5 using $f_{\theta_1}$ instead of $f_\alpha.$
$($b$)$ is also clear as $J^1 \cap K_{\theta_1} \supseteq \{ f_\alpha/D: \alpha < \theta_1     \}$
and $\langle f_\alpha/D: \alpha < \theta_1 \rangle$ is $<_1$-unbounded in $J^1$. Let's prove $($c$)$.
As $\{ s \in J^2: s <_1 s_*   \} \cap K_{\theta_1}=\emptyset$ and $I=J^1 \cup J^2$, we can easily see that
\[
J^1 \cap K_{\theta_1} = \{ s\in K_{\theta_1}: s <_1 s_*    \}.
\]
For each $i<\kappa$ set
\[
L_{i}=\{ s \in I_i:  s  <^2_i f_{\theta_1}(i) \text{~and~} s <^1_i h_*(i)           \}.
\]
It follows that $J^1 \cap K_{\theta_1}=\prod_{i<\kappa}L_{i}/D$,
 and so $J^1 \cap K_{\theta_1}$ is internal.
\end{proof}
It follows from the above claim that $J^1=\prod_{i<\kappa}L'_{i}/D$, where for $i<\kappa,$
$L'_i=\{ t \in I_i: \exists s \in L_i,~ t \leq^1_i s           \}$. Hence $J^1$ is internal
and so by Remark 4.2, $(J^1, J^2)$
is an internal cut of $I$ which completes the proof of Case 2.
The theorem follows.
\end{proof}

\section{$\mathscr{C}_{> 2^\kappa}(D)$ contains no non-symmetric pairs}
In this section we show that if $D$ is a uniform ultrafilter on $\kappa$, then
$\mathscr{C}_{>2^\kappa}(D)$ does not contain any non-symmetric pairs.
Again, we prove a stronger result from which the above claim, and hence Theorem 1.2$($b$)$ follows.
\begin{theorem}
Assume $D$ is a uniform ultrafilter on $\kappa,$ $\langle  I_i: i<\kappa     \rangle$ is a sequence of linear orders and $I=\prod_{i<\kappa} I_i/D.$
Also assume  $(J^1, J^2)$ is a cut of $I$ of cofinality $(\theta_1, \theta_2),$ where $\theta_1 \neq \theta_2$ are bigger than $2^\kappa$. Then the cut $(J^1, J^2)$ is internal.
\end{theorem}
\begin{proof}
Let
$<^1_i=<_{I_i}$ ($i<\kappa$) and $<_1=<_I.$ Let $<^2_i$, for $i<\kappa,$ be a well-ordering of $I_i$ with a last element and let $<_2$
be such that $(I, <_2) = \prod_{i<\kappa} (I_i, <^2_i) / D;$ so $<_2$ is a linear ordering of $I$ with a last element.

We say a sequence $\bar{K}=\langle  K_i: i< \kappa      \rangle$
catches $(J^1, J^2)$ if each $K_i \subseteq I_i$ is non-empty and for every $s_1 \in J^1$ and $s_2 \in J^2$
there is $t \in \prod_{i<\kappa}K_i/D$ such that $s_1 \leq_1 t \leq_1 s_2.$ Set
\[
S=\{ \bar{K}: \bar{K} \text{~catches~} (J^1, J^2)  \},
\]
and
\[
C=\{\bar{\mu}= \langle \mu_i: i<\kappa \rangle: \text{~There exists~} \bar{K} \in S\text{~such that~}\bigwedge_{i<\kappa} |K_i|=\mu_i    \}.
\]
We can define an order on $C$ by
\[
\bar{\mu}^1=\langle \mu^1_i: i<\kappa \rangle <_D \bar{\mu}^2=\langle \mu^2_i: i<\kappa \rangle \iff \{i<\kappa: \mu^1_i < \mu^2_i                \} \in D.
\]
Now consider the following statement:
\\
$(*)$$\hspace{4.cm}$ \underline{There is $\bar{\mu} \in C$ which is $<_D$-minimal.}

We consider two cases.
\\
{\underline{\bf Case 1. $(*)$ holds:}} Fix $\bar{\mu}=\langle \mu_i: i<\kappa \rangle$ witnessing $(*)$, and let $\bar{K} \in S$
be such that for all $i<\kappa,~ |K_i|=\mu_i.$
Let $<^3_i$ be a well-ordering of $K_i$ of order type $\mu_i$
and let $<_3$ be such that $(K, <_3)=\prod_{i<\kappa} (K_i, <^3_i)/D,$ where $K=\prod_{i<\kappa}K_i /D \subseteq I.$
Let $\theta_3 =\cf(\prod_{i<\kappa}\mu_i / D)$
and let $g_\alpha \in \prod_{i<\kappa}K_i$, $\alpha < \theta_3$,
be such that $\langle   g_\alpha / D: \alpha < \theta_3     \rangle$
is $<_3$-increasing and cofinal in $(K, <_3)$. As $\theta_1 \neq \theta_2,$ for some $l\in \{1,2\},$ $\theta_3 \neq \theta_l.$
Assume without loss of generality that $\theta_3 \neq \theta_1$.

For $\alpha < \theta_3$ and $i<\kappa$ set
\[
K_{\alpha, i}=\{ s \in K_i: s <^3_i g_\alpha(i)               \}.
\]
and
\[
K_\alpha = K \upharpoonright \{ s: s <_3 g_\alpha /D            \} = \prod_{i<\kappa} K_{\alpha, i}/D.
\]
Then the sequence $\langle   K_\alpha: \alpha < \theta_3   \rangle$ is $\subseteq$-increasing and $K=\bigcup_{\alpha < \theta_3}K_\alpha.$
The next claim is evident from our construction.
\begin{claim}
$K$ is internal.
\end{claim}
By our choice of $\bar{\mu},$ the sequence $ \langle  K_{\alpha, i}: i<\kappa    \rangle$
does not catch $(J^1, J^2)$,
and hence we can find
$s_\alpha \in J^1$ and $t_\alpha \in J^2$ such that
\[
K_{\alpha} \cap \{s \in I: s_\alpha <_1 s <_1 t_\alpha       \} =\emptyset.
\]
As $\theta_3 \neq \theta_1,$ there is $s_* \in J^1$
such that
\[
\sup\{\alpha < \theta_3: s_\alpha \leq_1 s_*         \} =\theta_3.
\]
It follows that $K \cap \{s \in J^1: s_* \leq_1 s   \}=\emptyset.$
As $K$ catches $(J^1, J^2),$ it follows that $J^2 \cap K$ is $<_1$-cofinal
in $J^2$ from below, and since $K$ is internal, the arguments of section 2 show that
$J^2$ is also internal, and hence by Remark 4.2, $(J^1, J^2)$ is an internal cut of $I$, as required.
\\
{\underline{\bf Case 2. $(*)$ fails:}} Clearly $<_D$ is a linear order on $C$, so it has a co-initiality, call it $\theta_3$.
As $(*)$ fails, $<_D$ is not well-founded and so $\theta_3 \geq \aleph_0$.

\begin{claim}
$\theta_3 \leq 2^\kappa.$
\end{claim}
\begin{proof}
Suppose not. Let $\langle \bar \mu_\xi: \xi < (2^\kappa)^+         \rangle$ be a $<_D$-decreasing chain of elements of $C$. Define a partition
$F: [(2^\kappa)^+]^2 \to \kappa$ by
\[
F(\xi, \zeta) = \min\{ i< \kappa: \mu^\zeta_i < \mu^\xi_i            \},
\]
which is well-defined as $\{ i< \kappa: \mu^\zeta_i < \mu^\xi_i            \} \in D,$ in particular it is non-empty.
By the Erd\"{o}s-Rado partition theorem, there are $X \subseteq (2^\kappa)^+$ of size $\kappa^+$ and some fixed $i_* < \kappa$ such that for all
$\xi < \zeta$ in $X$, $F(\xi, \zeta)=i_*.$ Thus
\[
\xi < \zeta \in X \implies \mu^\zeta_{i_*} < \mu^\xi_{i_*},
\]
which is impossible.
\end{proof}

Let $\langle  \bar{\mu}_\xi=\langle \mu_{\xi, i}: i<\kappa \rangle: \xi < \theta_3       \rangle$
be $<_D$-decreasing which is unbounded from below in $(C, <_D)$. For $\xi < \theta_3$ choose $\bar{K}_\xi = \langle K_{\xi, i}: i<\kappa     \rangle \in S$
such that for all $i<\kappa, ~ |K_{\xi, i}|=\mu_{\xi, i}.$
Let $K_\xi=\prod_{i<\kappa} K_{\xi,i}/D \subseteq I.$ We consider two subcases.
\\
{\underline{\bf Subcase 2.1. For some $\xi < \theta_3$, $~K_\xi \cap J^1$ is bounded in $(J^1, <_1)$:}}
Fix such a $\xi < \theta_3$ and let $s_* \in J^1$
be a bound. Then as $\bar{K}_\xi$ catches $(J^1, J^2)$, it follows that $K_\xi \cap J^2$
is unbounded in $J^2$ from below.  Since $K_\xi$ is internal and  $K_\xi \cap J^2$
is unbounded in $J^2$ from below, so  $J^2$ is internal. In follows that the cut $(J^1, J^2)$
is internal and we are done.
\\
{\underline{\bf Subcase 2.2. For all $\xi < \theta_3$, $~K_\xi \cap J^1$ is unbounded in $(J^1, <_1)$:}}
Since $\cf(J^1, <_1)=\theta_1$, there are functions $f_\alpha \in \prod_{i<\kappa}I_i,$ for $\alpha < \theta_1,$
such that
\begin{enumerate}
\item $\forall \alpha < \theta_1, f_\alpha /D \in J^1,$
\item  $\langle f_\alpha/D: \alpha < \theta_1     \rangle$ is $<_1$-increasing,
\item $\langle f_\alpha/D: \alpha < \theta_1     \rangle$ is a $<_1$-cofinal subset of $J^1$.
\end{enumerate}

For every $\alpha <\theta_1$ and $\xi < \theta_3$ there are $\beta$ and $g$ such that
\begin{enumerate}
\item [(4)] $\alpha < \beta < \theta_1,$

\item [(5)] $g \in \prod_{i< \kappa} K_{\xi, i}$,

\item [(6)] $f_\alpha / D <_1 g/D <_1 f_\beta/D.$

\end{enumerate}
For $\alpha <\beta < \theta_1$ set
\[
\Lambda_{\alpha, \beta} = \{   (\xi, i): f_\alpha(i) \leq^1_i f_\beta(i) \text{~and there is~} s \in K_{\xi, i} \text{~such that~} f_\alpha(i) \leq^1_i s \leq^1_i f_\beta(i)       \}.
\]
For $i<\kappa$ set $\Lambda_{\alpha, \beta, i}=\{ \xi < \theta_3: (\xi, i) \in \Lambda_{\alpha, \beta}         \}$
and  $\Xi_{\alpha, \beta} = \{ i< \kappa: \Lambda_{\alpha, \beta, i} \neq \emptyset    \}$.
Also let $F^1_{\alpha, \beta}, F^2_{\alpha, \beta}$ be functions with domain $\kappa$
such that
\begin{itemize}
\item If $i \in \Xi_{\alpha, \beta},$ then
\[
F^1_{\alpha, \beta}(i) = \min\{  \mu_{\xi, i}: (\xi, i) \in \Lambda_{\alpha, \beta}      \},
\]
and
\[
F^2_{\alpha, \beta}(i) = \min\{ \xi: \mu_{\xi, i} =F^1_{\alpha, \beta}(i)     \}.
\]

\item If $i \in \kappa \setminus \Xi_{\alpha, \beta},$ then $F^1_{\alpha, \beta}(i)= F^2_{\alpha, \beta}(i)=0.$

\end{itemize}
\begin{claim}
For each $\alpha < \theta_1$ there exist $A_\alpha \subseteq (\alpha, \theta_1),$ functions $F^1_\alpha, F^2_\alpha$
and a set $\Xi_\alpha$ such that
\begin{enumerate}
\item $A_\alpha = \{ \beta \in (\alpha, \theta_1): F^1_{\alpha, \beta}=F^1_\alpha, F^2_{\alpha, \beta}=F^2_\alpha$ and $\Xi_{\alpha, \beta}=\Xi_\alpha     \}$.

\item $\sup(A_\alpha)=\theta_1.$
\end{enumerate}
\end{claim}
\begin{proof}
As $\theta_3 \leq 2^\kappa,$ we have
\begin{center}
 $|\{(F^1_{\alpha, \beta}, F^2_{\alpha, \beta}, \Xi_{\alpha, \beta}): \alpha < \beta < \theta_1  \}| \leq \theta_3^\kappa =2^\kappa < \theta_1$.
\end{center}
 So there is an unbounded subset $A_\alpha$ of $ \theta_1$ such that all tuples $(F^1_{\alpha, \beta}, F^2_{\alpha, \beta}, \Xi_{\alpha, \beta}), \beta \in A_\alpha$,
 are the same. The result follows immediately.
\end{proof}
 The next claim can be proved in a similar way.
\begin{claim}
There are $A \subseteq \theta_1$, functions $F_1, F_2$ and a set $\Xi$ such that
\begin{enumerate}
\item $A = \{ \alpha < \theta_1: F^1_{\alpha}=F_1, F^2_{\alpha}=F_2$ and $\Xi_{\alpha}=\Xi    \}$.

\item $\sup(A)=\theta_1.$
\end{enumerate}
\end{claim}
Let $\bar{K}^*=\langle  K^*_i: i<\kappa          \rangle$
where $K^*_i=K_{F_2(i), i}$ and let $\bar{\mu}^*= \langle   \mu^*_i: i< \kappa     \rangle$
be defined by $\mu^*_i=|K^*_i|$. Note that
\[
\mu^*_i=|K^*_i| = |K_{F_2(i), i}|=\mu_{F_2(i), i}=F_1(i).
\]

\begin{claim}
For every $\xi < \theta_3,$  $~\bar{\mu}^* \leq_D \bar{\mu}_\xi$.
\end{claim}
\begin{proof}
Choose $\alpha \in A$ and $\beta \in A_\alpha.$ So we have $F_1=F^1_{\alpha, \beta}, F_2=F^2_{\alpha, \beta}$
and $\Xi=\Xi_{\alpha, \beta}.$
By the  construction, there is $t \in K_\xi$
such that $f_\alpha /D <_1 t <_1 f_\beta /D$. Let $t=g/D, $ where $g \in \prod_{i<\kappa}K_{\xi, i}.$ Then
\begin{center}
$f_\alpha(i) <^1_i g(i) <^1_i f_\beta(i) \Rightarrow (\xi, i) \in \Lambda_{\alpha, \beta}$
$\Rightarrow \mu^*_i= F_1(i)=F^1_{\alpha, \beta}(i) \leq \mu_{\xi, i}$.
\end{center}
So
\[
\{ i< \kappa: \mu^*_i \leq \mu_{\xi, i}    \} \supseteq \{ i< \kappa:   f_\alpha(i) <^1_i g(i) <^1_i f_\beta(i)   \} \in D,
\]
and the result follows.
\end{proof}
\begin{claim}
$\bar{\mu}^* \in C.$
\end{claim}
\begin{proof}
We show that $\bar{K}^*$ catches $(J^1, J^2)$, so that $\bar{K}^* \in S$ witnesses   $\bar{\mu}^* \in C.$
So let $s_1 \in J^1$ and $s_2 \in J^2.$ Pick $\alpha \in A$
such that $s_1 <_1 f_\alpha/D.$ Let $\beta \in A_\alpha.$
By our construction there is $g \in \prod_{i<\kappa} K^*_i$
such that $f_{\alpha}/D <_1 g/D <_1 f_{\beta} /D$ and hence
\[
s_1 <_1 f_{\alpha}/D <_1 g/D <_1 f_{\beta} /D <_1 s_2.
\]
The claim follows.
\end{proof}
But Claims 5.6 and 5.7 give us a contradiction to the choice of the sequence
$\langle  \bar{\mu}_\xi: \xi < \theta_3       \rangle$. This contradiction finishes the proof.
\end{proof}

School of Mathematics, Institute for Research in Fundamental Sciences (IPM), P.O. Box:
19395-5746, Tehran-Iran.

E-mail address: golshani.m@gmail.com

Einstein Institute of Mathematics, The Hebrew University of Jerusalem, Jerusalem,
91904, Israel, and Department of Mathematics, Rutgers University, New Brunswick, NJ
08854, USA.

E-mail address: shelah@math.huji.ac.il

\end{document}